\documentclass{amsart}

\usepackage{amssymb,amsmath,amsfonts,fancyhdr,amsthm}
\usepackage{graphicx}
\usepackage[dvips,all]{xy}
\UseComputerModernTips

\newtheorem{thm}{Theorem}[section]

\newtheorem{conj}[thm]{Conjecture}
\newtheorem{qst}[thm]{Question}

\theoremstyle{definition}

\theoremstyle{remark}
\newtheorem{rmk}[thm]{Remark}

\newcommand{\Z}{\mathbb Z}

\renewcommand{\phi}{\varphi}

\newcommand{\ann}{\mathop{\mathrm{Ann}}\nolimits}

\newcommand{\reg}{\mathop{\mathrm{reg}}\nolimits}

\newcommand{\pd}{\mathop{\mathrm{pd}}\nolimits}

\newcommand{\tensor}{\otimes}
\newcommand{\tor}{\mathop{\mathrm{Tor}}\nolimits}
\renewcommand{\hom}{\mathop{\mathrm{Hom}}\nolimits}

\newcommand{\lto}[1]{\overset{#1}{\longrightarrow\,}}
\newcommand{\xto}[1]{\overset{#1}{\to}}

\title[Stillman's Question and Herzog's Conjecture]{Stillman's Question for Exterior Algebras and Herzog's Conjecture on Betti Numbers of Syzygy Modules}
\author{Jason McCullough}
\email{jmccullo@iastate.edu}
\address{Iowa State University, Department of Mathematics, Ames, IA 50011}


\begin{document}
\maketitle

\begin{abstract} Let $K$ be a field of characteristic $0$ and consider exterior algebras of finite dimensional $K$-vector spaces.  In this short paper we exhibit principal quadric ideals in a family whose Castelnuovo-Mumford regularity is unbounded.  This negatively answers the analogue of Stillman's Question for exterior algebras posed by I. Peeva.  We show that these examples are dual to modules over polynomial rings that yield counterexamples to a conjecture of J. Herzog on the Betti numbers in the linear strand of syzygy modules.
\end{abstract}

\section{Introduction}

Let $K$ be a field and let $S$ be the symmetric algebra of some finite dimensional $K$-vector space.  Stillman \cite[Problem 3.14]{PS} posed the following question: Can the projective dimension $\pd_S(S/I)$ of a homogeneous ideal $I$ be bounded purely in terms of the number and degrees of the minimal generators of $I$?  Caviglia showed that this question was equivalent to the parallel question where one replaces projective dimension by regularity.  (cf. \cite[Theorem 2.4]{MS})  Ananyan-Hochster \cite{AH} recently gave a positive answer to this question in full generality.  


Now let $E$ be the positively graded exterior algebra of a finite dimensional $K$-vector space.  While resolutions over $E$ need not be finite, the regularity $\reg_E(M)$ of a finitely generated $E$-module is finite since $E$ is a Koszul algebra.  Here $\reg_E(M)$ is defined as
\[\reg_E(M) \mathrel{\mathop:}= \sup\{j - i | \tor_i^E(M,K)_j \neq 0\}.\]
Irena Peeva posed the following variant of Stillman's Question at the Joint Introductory Workshop at MSRI in the fall of 2012:

\begin{qst}[Peeva]\label{peeva} Can the regularity $\reg_E(E/I)$ of a homogeneous ideal $I$ of $E$ be bounded purely in terms of the number and degrees of the generators?
\end{qst}

Surprisingly, and contrary to the symmetric algebra case, the answer to this question is no.  In Section~\ref{stillman} we present a family of principal quadric ideals in exterior algebras over an arbitrary field whose regularity is unbounded.  

Now let $M$ be a finitely generated, graded $S$-module where $S = K[x_1,\ldots,x_n]$.  Let $d$ denote the minimal degree of a generator of $M$.  We consider the Betti numbers in the linear strand of $M$, that is:
\[\beta_i^{\text{lin}}(M) \mathrel{\mathop:}= \beta_{i,d+i}(M),\]
where $\beta_{i,j}(M) = \dim_K \tor_i^S(M,K)_j$.  
The length of the linear strand of $M$ is $\max\{i\,|\,\beta_i^{\text{lin}}(M) \neq 0\}$.  Herzog proposed the following lower bound on the Betti numbers in the linear strand of $k$th syzygy modules.
\begin{conj}[{Herzog \cite{Herzog}}]\label{her}
If $M$ is a graded $k$th syzygy module over $S$ with linear strand of length $p$, then 
\[\beta_{i}^{\text{lin}}(M) \ge \binom{p+k}{i+k}.\]
\end{conj}

\noindent The conjecture has been proved in the following cases:
\begin{enumerate}
\item Herzog proved the case $k = 0$ in \cite{Herzog}.
\item Herzog was motivated by a result of Green \cite{Green} which contained the case $k = 1, i = 0$.  See also similar results by Eisenbud-Koh \cite{EK}.
\item Reiner and Welker \cite{RW} proved the case when $M$ is a monomial ideal and $k = 1$.
\item R\"omer \cite{Romer} proved the following cases:
\begin{enumerate}
\item When $k = 1$ and $0 \le i \le p$.
\item When $p > 0$ and $i = p-1$; i.e. $\beta_{p-1}^{\text{lin}}(M) \ge \binom{p+k}{p-1+k} = p + k$.
\item When $M$ is $\Z^n$-graded, R\"omer proved the conjecture in full generality.

\end{enumerate}
\end{enumerate}

By modifying a recent construction of Iyengar-Walker \cite{IyengarWalker}, we use the Bernstein-Gel'fand-Gel'fand (BGG) correspondence to produce counterexamples for virtually all other cases of Herzog's Conjecture.  More precisely, in Section~\ref{herzogcounter} we construct for each $n \ge 1$ a finitely generated, graded $S = K[x_1,\ldots,x_{2n}]$-module $M$ that is an $n$th sygyzy module such that $M$ has linear resolution of length $n$ and has graded Betti numbers
\[\beta_i^{\text{lin}}(M) = \binom{2n}{n+i} - \binom{2n}{n+i+2}, \qquad \text{for } 0 \le i \le n.\]
We note that for $0 \le i \le n-2$ that
\[\beta_i^{\text{lin}}(M) = \binom{2n}{n+i} - \binom{2n}{n+i+2} <  \binom{n+n}{n+i},\]
contradicting the conjecture.  These examples are BGG dual to the principal quadric $E$-ideals mentioned above.

We note that a different construction of Conca-Herbig-Iyengar \cite{ConcaHerbigIyengar} also gives a family of counterexamples to Conjecture~\ref{her}, although that was not their aim.  Specifically, they show \cite[Theorem 5.1]{ConcaHerbigIyengar} that if $S = K[x_1,\ldots,x_n,y_1,\ldots,y_n]$ and $I$ is the ideal
\[I = \left( \{x_iy_j, x_iy_j - x_jy_i\,|\, 1 \le i,j \le n \text{ with } i \neq j\}\right),\]
then the cyclic module $B = S/I$ has graded Betti numbers
\[\beta_i^S(B) = \begin{cases}
1 & \text{ for } i,j = 0\\
\binom{2n}{i+1} - \binom{2n}{i -1} - 2 \binom{n}{i+1} & \text{ for } 1 \le i \le n-1 \text{ and } j = i + 1\\
\binom{2n}{i} - \binom{2n}{i + 2} & \text{ for } n \le i \le 2n \text{ and } j = i + 2\\
0 & \text{otherwise}
\end{cases}
\]
In particular, setting $N = \mathrm{Syz}_n(B)$ gives an $n$th syzygy module with the same graded Betti numbers (after a degree shift) as $M$ above.

\section{The BGG Correspondence}

We briefly recall the Bernstein-Gel'fand-Gel'fand (BGG) correspondence.  We refer the reader to \cite{Eisenbud} or \cite{EFS} for more details.  Let $K$ be a field of characteristic $0$.  Fix a positive integer $n$.  Let $S = K[x_1,\ldots,x_{2n}]$ denote a standard graded polynomial ring over $K$.  Set $W = S_1$ and let $V = \hom_K(W,K)$ denote the vector-space dual of $W$.  Let $E = \wedge V$ be the exterior algebra of $V$.  Let $e_1,\ldots,e_{2n}$ be a dual basis to $x_1,\ldots,x_{2n}$.  Thus $E = K\langle e_1,\ldots,e_{2n}\rangle$, which we also view as a positively graded ring with $\deg(e_i) = 1$ to keep the notation in this paper consistent.  (In \cite{Eisenbud} and \cite{EFS}, $E$ is viewed as a negatively graded ring.)

Let $\mathbf{L}$ denote the functor from the category of graded $E$-modules to the category of linear free complexes of $S$-modules defined as follows: Given any graded $E$-module $P$, we define $\mathbf{L}(P)$ to be $S \tensor_K P$ viewed as a complex of graded free $S$-modules, where $\mathbf{L}(P)_i = S \tensor_K P_{-i}$ and the differential is induced by
\[
\begin{tabular}{ccccc}
$\mathbf{L}(P):  \cdots \to$ &$S \tensor_K P_{-i} $&$\xto{d_i}$ &$S \tensor_K P_{1-i} $&$\to \cdots$ \\
& $s \tensor p$ &$\mapsto$ &$\sum sx_j \tensor e_jp$ &
\end{tabular}
\]
We extend $\mathbf{L}$ to a functor on complexes of $E$-modules by taking the total complex of the resulting double complex.  There is an adjoint function $\mathbf{R}$ from complexes of $S$-modules to complexes of $E$-modules, defined in a similar way, which creates an equivalence of bounded derived categories of graded $S$-modules and that of graded $E$-modules.  An important property of the equivalence for the purpose of this paper is that under this equivalence the functor $\mathbf{L}$ identifies finitely generated $E$-modules with linear free $S$-complexes.  

\section{Counterexamples to Herzog's Conjecture}\label{herzogcounter}

First we recall a well-known fact about exterior algebras:

\begin{thm}[{cf. \cite[Theorem 5.2]{Snellman}}]\label{generic} Let $f \in E_2 = \wedge^2 V$ be a generic element.  Then the linear transformation
\[\wedge^r V \xto{f\cdot} \wedge^{r+2} V\]
is injective for $r \le n-1$ and is surjective for $r \ge n-1$.
\end{thm}

As remarked in \cite{IyengarWalker}, it suffices to pick $f = e_1e_2 + \cdots + e_{2n-1}e_{2n}$.  We now come to the first main result.  By convention, we set $\binom{m}{k} = 0$ if $k < 0$ or if $k > m$.

\begin{thm}\label{main}
Fix $n > 0$ and a field $K$ of characteristic $0$.  Let $S = K[x_1,\ldots,x_{2n}]$ be a standard graded polynomial ring.  There exists a finitely generated, graded $S$-module $M$, generated by degree $0$ elements such that $\reg(M) = 1$ and the graded Betti numbers are 
\[
\beta_{i,j}(M) = \begin{cases}     \binom{2n}{i+1}- \binom{2n}{i-1}& \text{ if } i = j \text{ and } 0 \le i \le n\\
\binom{2n}{i}-\binom{2n}{i+2}& \text{ if } j = i+1 \text{ and } n+1 \le i \le 2n\\
0 & \text{otherwise.}
\end{cases}
\]
In particular, if we set $N = \mathrm{Syz}_n(M)$, then $\beta_{i}^{lin}(N) = \binom{2n}{n+i} - \binom{2n}{n+i+2}$, contradicting Conjecture~\ref{her} for $0 \le i \le n-2$.
\end{thm}

\begin{proof} Let $f \in E_2$ be a general element.  Consider the graded complex of free $E$-modules
\[\mathbf{G_\bullet}: E(2-2n) \lto{f} E(-2n).\]
Let $\mathbf{F_\bullet} = \mathbf{L}(\mathbf{G}_\bullet)$.  Since $\mathbf{L}(E)$ is the usual Koszul complex $K_\bullet(\mathbf{x},S)[2n]$ with a shift in homological degree, we have $\mathbf{F_\bullet} = \mathrm{cone}(K_\bullet(\mathbf{x},S)[2] \to K_\bullet(\mathbf{x},S))$.  

The complex $\mathbf{F}_\bullet$ fits in a short exact sequence
\[0 \to K_\bullet(\mathbf{x},S)[1] \to \mathbf{F_\bullet} \to K_\bullet(\mathbf{x},S) \to 0.\]
Since $H_0(K_\bullet(\mathbf{x},S)) \cong K$ and $H_i(K_\bullet(\mathbf{x},S)) = 0$ for all $i > 0$, if follows that
\[H_i(\mathbf{F}_\bullet) = \begin{cases} K & \text{ for } i = 0, -1\\
0 & \text{ otherwise.}\\
\end{cases}
\]

By Theorem~\ref{generic}, every possible cancelation in $\mathbf{F_\bullet}$ that can occur does occur.  Thus the minimal subcomplex $\mathbf{F_\bullet}'$ of $\mathbf{F_\bullet}$ has graded Betti table

\[
\begin{tabular}{r|cccccccccccc}
       &-1&0&1&$\cdots$&n-1&n&$\cdots$&2n-1&2n\\
       \hline
       \text{0:}&1&2n&$\binom{2n}{2} - 1$&$\cdots$&$\binom{2n}{n} - \binom{2n}{n-2}$&\text{-}&\text{-}&\text{-}&\text{-}\\\text{1:}&\text{-}&\text{-}&\text{-}&\text{-}&\text{-}&$\binom{2n}{n} - \binom{2n}{n+2}$&$\cdots$&2n&1\\\end{tabular}
\]
\end{proof}
Note that under the BGG correspondence, the first linear strand corresponds to $\mathbf{L}(Ann_E(f))$ and the second linear strand is $\mathbf{L}(E/(f))$.  In particular, Since $H_i(\mathbf{F}_\bullet') = 0$ for $i > 0$, if we set $\mathbf{F}_{\ge 0}'$ to be the truncation of $\mathbf{F}_\bullet'$, then  $\mathbf{F}_{\ge 0}'$ is the minimal free resolution of an $S$-module $M$.  (We shift graded degrees to ensure that $M$ is generated in degree $0$.)  In particular, $M$ now has graded Betti table of the form
\[
\begin{tabular}{r|cccccccccccc}
       &0&1&$\cdots$&n-1&n&$\cdots$&2n-1&2n\\
       \hline
       \text{0:}&2n&$\binom{2n}{2}-1$&$\cdots$&$\binom{2n}{n} - \binom{2n}{n-2}$&\text{-}&\text{-}&\text{-}&\text{-}\\\text{1:}&&\text{-}&\text{-}&\text{-}&$\binom{2n}{n} - \binom{2n}{n+2}$&$\cdots$&2n&1\\\end{tabular}
\]
Setting $N = \mathrm{Syz}_n(M)$, we see that $N$ is an $n$th syzygy with linear free resolution and graded Betti numbers as prescribed above.

\section{Peeva-Stillman Question for Exterior Algebras}\label{stillman}

Now let $E$ be the exterior algebra of a finite dimensional $K$-vector space.  Here we set the exterior variables to have degree $+1$ rather than $-1$, as in the previous section.  While resolutions over $E$ need not be finite, since $E$ is a Koszul algebra, the regularity $\reg_E(M)$ of a finitely generated $E$-module is finite.  Here $\reg_E(M)$ is defined as
\[\reg_E(M) := \sup\{j - i | \tor_i^E(M,K)_j \neq 0\}.\]
Irena Peeva posed the following variant of Stillman's Question at the Joint Introductory Workshop at MSRI in the Fall of 2012:

The following theorem gives the promised negative answer to Question~\ref{peeva}.


\begin{thm}\label{extmain} Let $K$ be a field of characteristic $0$.  Fix any positive integer $n$.  Let $V$ be the $K$-vector space with basis $e_1,e_2,\ldots,e_{2n}$ and consider the exterior algebra $E = \bigwedge V$.  Set $f = e_1e_2 + e_3e_4 + \cdots + e_{2n-1}e_{2n}$.  Then $\reg_E(E/(f)) = n$.
\end{thm}

\begin{proof}
Consider the minimal free resolution of $E/(f)$ over $E$.  Note that we may identify $\mathrm{Syz}_2(E/(f))$ with $\ann_E(f)$.  By Theorem~\ref{generic}, $\ann_E(f)$ has no elements of degree $< n$.  Thus $\reg_E(E/(f)) \ge n$.  On the other hand, by the proof of Theorem~\ref{main}, we see that $\mathbf{L}(E/(f))$ has a linear free $S$-resolution.  By the Dictionary Theorem \cite[Theorem 7.7]{Eisenbud}, $\hom_K(E/(f), K) \cong \ann_E(f)$ has a linear free $E$-resolution.  Therefore $\reg_E(E/(f)) = n$.
\end{proof}

\begin{rmk} The previous theorem also holds over any field of positive characteristic $p > 0$, but it is no longer true that $\ann(f)$ has a linear free resolution since in particular $f^p = 0$.
\end{rmk}

\section{An Example}

Finally we show how one can construct $n$th syzygy modules with linear free resolutions of length $n$ and with graded Betti numbers smaller than those predicted by Herzog's Conjecture in Macaulay2.  Here we make use of the BGG package written by Abo, Decker, Eisenbud, Schreyer, Smith and Stillman.  Here we demonstrate the case $n = 4$.

\begin{verbatim}
i1 : n = 4;

i2 : E = QQ[e_1..e_(2*n),SkewCommutative=>true]; 
-- an exterior algebra in 2n variables

i3 : S = QQ[x_1..x_(2*n)];                       
-- a polynomial ring in 2n variables

i4 : I = ideal sum for i from 1 to n list e_(2*i-1)*e_(2*i)

o4 = ideal(e e  + e e  + e e  + e e )
            1 2    3 4    5 6    7 8

o4 : Ideal of E

i5 : betti res I   
-- The resolution over E of E/I. reg(E/I) = n

            0 1  2   3    4    5    6     7     8     9
o5 = total: 1 1 42 288 1155 3520 9009 20384 42042 80640
         0: 1 .  .   .    .    .    .     .     .     .
         1: . 1  .   .    .    .    .     .     .     .
         2: . .  .   .    .    .    .     .     .     .
         3: . .  .   .    .    .    .     .     .     .
         4: . . 42 288 1155 3520 9009 20384 42042 80640

o5 : BettiTally

i6 : J = ann I;

o6 : Ideal of E

i7 : P = J / ideal product flatten entries vars E;

i8 : loadPackage "BGG"

o8 = BGG

o8 : Package

i9 : F = bgg(2*n-2,P,S);

             8       27
o9 : Matrix S  <--- S

i10 : M = coker F  ** S^{-(2*n-1)};  
-- M is an S-module

i11 : betti res M   
-- the linear strand of the resolution of M is L(J/(e_1*...*e_2n))

             0  1  2  3  4  5  6 7 8
o11 = total: 8 27 48 42 42 48 27 8 1
          0: 8 27 48 42  .  .  . . .
          1: .  .  .  . 42 48 27 8 1

o11 : BettiTally

i12 : G = bgg(n-1,comodule I,S);

              42       48
o12 : Matrix S   <--- S

i13 : N = coker G  ** S^{-n};  
-- N is the nth syzygy module of M, with a shift in grading

i14 : betti res N    
-- its resolution is L(E/I) under the BGG correspondence

              0  1  2 3 4
o14 = total: 42 48 27 8 1
          0: 42 48 27 8 1

o14 : BettiTally\end{verbatim}

It is clear that this is not the most extreme counterexample to Herzog's Conjecture.  Note that the module $N$ above has rank $14$ and is a $4$th syzygy module over $S$.  By a Theorem of Bruns (cf. \cite[Corollary 3.13]{EG}), there exists a free submodule $F \subset N$ of rank $10$ such $N/F$ is a still a $4$th syzygy module.  The module $N/F$ would then have Betti table
\[
\begin{tabular}{r|ccccc}
       &0&1&2&3&4\\
       \hline
       \text{0:}&32&48&27&8&1\\\end{tabular}
\]
It would be interesting to know in general what the minimal Betti numbers are for a $n$th syzygy module with linear free resolution.

\section*{Acknowledgements} 

The author thanks Srikanth Iyengar, Irena Peeva, and Mark Walker for useful conversations concerning this paper and also Jerzy Weyman for pointing out reference \cite{Snellman}.

\bibliographystyle{amsplain}
\bibliography{bib}

\end{document}